\documentclass[a4paper,12pt,reqno]{article}  
\usepackage{graphicx}
\usepackage{amsmath}
\usepackage{amssymb}
\usepackage{enumerate}
\usepackage{amsthm}
\usepackage[T1]{fontenc}
\usepackage[latin1]{inputenc}
\usepackage[round]{natbib}

\newtheorem{theorem}{Theorem}
\newtheorem{proposition}[theorem]{Proposition}
\newtheorem{corollary}[theorem]{Corollary}
\newtheorem{lemma}[theorem]{Lemma}
\theoremstyle{remark}
\newtheorem{remark}[theorem]{Remark}
\newtheorem{example}[theorem]{Example}
\theoremstyle{definition}

\newcommand{\pr}[2]{P_{#1}\left[#2\right]}

\newcommand{\E}[2]{E_{#1}\left[#2\right]}

\newcommand{\abs}[1]{\left\vert#1\right\vert}

\newcommand{\set}[1]{\left\lbrace #1 \right\rbrace}

\providecommand{\norm}[1]{\left\lVert#1\right\rVert}

\begin{document}
\title{{\sc Survival probabilities of weighted random walk}}
\author{\renewcommand{\thefootnote}{\arabic{footnote}} {\sc Frank Aurzada}\footnotemark[1]\\
\renewcommand{\thefootnote}{\arabic{footnote}}{\sc Christoph Baumgarten}\footnotemark[1]}
\date{\today}

\footnotetext[1]{
Technische Universit\"at Berlin, Institut f\"ur Mathematik, Sekr.\ MA 7-4,
Stra{\ss}e des 17.\ Juni 136, 10623 Berlin, Germany,
{\sl aurzada@math.tu-berlin.de}, 
{\sl baumgart@math.tu-berlin.de}
}
\maketitle

\begin{abstract}
We study the asymptotic behaviour of the probability that a weighted sum of centered i.i.d.\ random variables $X_k$ does not exceed a constant barrier.\\
For regular random walks, the results follow easily from classical fluctuation theory, while this theory does not carry over to weighted random walks, where essentially nothing seems to be known.\\
First we discuss the case of a polynomial weight function and determine the rate of decay of the above probability for Gaussian $X_k$. This rate is shown to be universal over a larger class of distributions that obey suitable moment conditions.\\
Finally we discuss the case of an exponential weight function. The mentioned universality does not hold in this setup anymore so that the rate of decay has to be determined separately for different distributions of the $X_k$. We present some results in the Gaussian framework.
\end{abstract}

\section{Introduction}
\subsection{Statement of the problem}
In this article we study the asymptotic behaviour of
\begin{equation}\label{eq:one-sided-exit-problem}
   \pr{}{\sup_{0 \leq t \leq T} Z_t \leq 1}, \quad \text{or} \quad \pr{}{\sup_{n=1,\dots,N} Z_n \leq 0}, 
\end{equation}
as $T,N \to \infty$ for a certain class of stochastic processes $Z = (Z_t)_{t \geq 0}$ to be specified below. The probability above is often called survival probability up to time $T$ (also persistence probability). The problem of determining the asymptotic behaviour of \eqref{eq:one-sided-exit-problem} is sometimes also called one-sided exit problem since the survival probability can also be expressed using first hitting times. Typically, it cannot be computed explicitly. For most processes considered here, it decays polynomially with time (ignoring terms of lower order) , i.e.
\[
   \pr{}{\sup_{0 \leq t \leq T} Z_t \leq 1} = T^{-\theta + o(1)}, \quad T \to \infty,
\]
where $\theta$ is called the survival exponent.\\
Apart from pure theoretical interest in this classical problem, research on survival probabilities of integrated processes was motivated by the investigation of the inviscid Burgers equation, see e.g.\ \cite{sinai:1992, bertoin:1998, molchan-khokhlov:2004}. Further motivations are pursuit problems and a relation to questions about random polynomials; we refer to \cite{li-shao:2004} for a recent overview of applications. We mention that the problem of determining the survival exponent is relevant in various physical models such as reaction diffusion systems, granular media and Lotka-Volterra models for population dynamics, see the survey of \cite{majumdar:1999} with a collection of examples.\\
Although \eqref{eq:one-sided-exit-problem} is a classical problem, it has not been studied very intensively so far except for a few Gaussian processes and the case of processes having independent and stationary increments such as random walks and L\'{e}vy processes. The latter results are part of classical fluctuation theory. In the present article, we drop the assumption of stationary increments and study deterministically weighted sums of i.i.d.\ random variables. For such processes, there is virtually no theory available so far. \\
Our approach focusses on the analysis of the case of Gaussian increments first. Then universality results are shown by transferring the statement from Gaussian to more generally distributed increments.\\
The article is organized as follows. In Section~\ref{sec:mainres}, we introduce the class of processes in detail and summarize the main results. Some related work on survival probabilities is reviewed in Section~\ref{sec:related}. We discuss the exit problem for Gaussian weighted random walks in Section~\ref{sec:gaussian}. Here, the cases of a polynomially, a subexponentially, and an exponentially increasing weight function are considered in separate subsections. In Section~\ref{sec:weighted_rw}, the results of the Gaussian case for a polynomial weight function are extended to a broader class of weighted random walks whose increments obey certain moment conditions.\\
Finally, let us introduce some notation: If $f,g: \mathbb{R} \to \mathbb{R}$ are two functions, we write $f \precsim g$ if $\limsup_{x \to \infty} f(x)/g(x) < \infty$ and $f \asymp g$ if $f \precsim g$ and $g \precsim f$. Moreover, $f \sim g$ if $f(x)/g(x) \to 1$ as $x \to \infty$. 
\subsection{Main results} \label{sec:mainres}
We investigate the behaviour of survival probabilites of processes $Z=(Z_n)_{n \geq 1}$ defined by
\begin{equation}\label{def:weighted_RW}
   Z_n := \sum_{k=1}^n \sigma(k) X_k, \quad n \geq 1,
\end{equation}
where $X_1, X_2,\dots$ are i.i.d.\ random variables such that $\E{}{X_1} = 0$ and $\sigma \colon [0,\infty) \to (0, \infty)$ is a measurable function. We call $Z$ a weighted random walk with weight function $\sigma$. \\
Despite the obvious resemblance, the methods for computing the survival probability of (unweighted) random walks ($\sigma(n)\equiv 1$) do not carry over since they strongly rely upon the stationarity of increments that allows for an explicit computation of the generating function of the first hitting time of the set $(0,\infty)$.\\
Note that if the $X_k$ have a standard normal distribution, then the processes $(Z_n)_{n \geq 1}$ and $(B_{\kappa(n)})_{n \geq 1}$ have the same law where $\kappa(n) := \sigma(1)^2 + \dots + \sigma(n)^2$ and $B$ is a standard Brownian motion. Therefore, the computation for the weighted Gaussian random walk reduces to the case of Brownian motion evaluated at discrete time points. In this setup, we prove the following theorem.
\begin{theorem}\label{thm:overview_1}
Let $\kappa \colon [0,\infty) \to (0,\infty)$ be a measurable function such that $\kappa(N) \asymp N^q$ for some $q >0$. If there is some $\delta < q$ such that $\kappa(N+1) - \kappa(N) \precsim N^\delta$, then
\[
   \pr{}{\sup_{n=1,\dots,N} B_{\kappa(n)} \leq 0} = N^{-q/2 + o(1)}, \quad N \to \infty.
\]
\end{theorem}
The lower order term $N^{o(1)}$ can be specified more precisely (Theorem~\ref{thm:polynomial_case_BM}). In particular, we have under the assumptions of Theorem~\ref{thm:overview_1} that
\begin{equation}\label{eq:discr_and_cont_time}
   \pr{}{\sup_{n=1,\dots,N} B_{\kappa(n)} \leq 0} =  \pr{}{\sup_{t \in [1,\kappa(N)]} B_{t} \leq 0} \, N^{o(1)} = N^{-q/2 + o(1)}, \quad N \to \infty.
\end{equation}
In the Gaussian framework, the weight function $\sigma(n) = n^p$ corresponds to $\kappa(n) = \sum_{k=1}^n \sigma(k)^2 \asymp n^{2p+1}$ as remarked above. This implies that the survival exponent for the weighted Gaussian random walk $Z$ is equal to $\theta = p+1/2$.\\
In fact, we show that this survival exponent is universal over a much larger class of weighted random walks in case the $X_k$ are not necessarily Gaussian:
\begin{theorem}\label{thm:overview_2}
   Let $(X_k)_{k\geq 1}$ be a sequence of i.i.d.\ random variables with $\E{}{X_1} = 0$, $\E{}{X_1^2} > 0$ and $\E{}{e^{a \abs{X_1}}} < \infty$ for some $a > 0$. If $\sigma$ is increasing and $\sigma(N) \asymp N^{p}$, then for the weighted random walk $Z$ defined in \eqref{def:weighted_RW}, we have
\[
   \pr{}{\sup_{n=0,\dots,N} Z_n \leq 0} = N^{-(p +1/2) + o(1)}, \quad N \to \infty.
\]
\end{theorem}
The proof of the lower bound for the survival probability in Theorem~\ref{thm:overview_2} under weaker assumptions (Theorem~\ref{thm:weighted_rw_lower_bound}) is based on the Skorokhod embedding. The upper bound (Theorem~\ref{thm:weighted_rw_upper_bound}) is established using a coupling of \cite{k-m-t:1976}. In either case, the problem is reduced to finding the survival exponent for Gaussian increments, i.e.\ to the case treated in Theorem~\ref{thm:overview_1}.

As noted in (\ref{eq:discr_and_cont_time}), Theorem~\ref{thm:overview_1} shows that it does not matter for the asymptotic behaviour of the survival exponent whether one samples the Brownian motion at the discrete points $(\kappa(n))_{n\geq 1}$ or over the corresponding interval if $\kappa$ increases polynomially. This result can be generalized to functions of the type $\kappa(n)=\exp(n^\alpha)$, $n \geq 0$, at least for $\alpha<1/4$ (Theorem~\ref{thm:subexp}). This fact turns out to be wrong however for the case $\alpha=1$ in general. Namely, if we consider an exponential function $\kappa(n) = \exp(\beta n)$ for $n\geq 0$ and some $\beta > 0$, it follows from Slepian's inequality in the Gaussian case that
\[
   \lim_{N \to \infty} -\frac{1}{N} \log \pr{}{\sup_{n=0,\dots,N} B\left(e^{\beta n}\right) \leq 0} =: \lambda_\beta
\]
exits for every $\beta>0$, and that $\beta\mapsto\lambda_\beta$ is increasing. However, one has
\[
   \lambda_\beta < \beta/2 = \lim_{N \to \infty} -\frac{1}{N} \log \pr{}{\sup_{t \in [0,N]} B\left(e^{\beta t}\right) \leq 0}
\]
at least for $\beta > 2 \log 2$ showing that the rates of decay in the discrete and continuous time framework do not coincide in contrast to \eqref{eq:discr_and_cont_time}. Additionally, the rate of decay of the survival probability for an exponentially weighted random walk now depends on the distribution of the $X_k$ even under exponential moment conditions, that is, a universality result similar to the polynomial case found in Theorem~\ref{thm:overview_2} does not hold.\\
In the Gaussian case, we state upper and lower bounds on the rate of decay in Theorem~\ref{thm:summary_exp} and characterize $\lambda_\beta$ as an eigenvalue of a certain integral operator in Proposition~\ref{prop:int_eq}. Unfortunately, an explicit computation of $\lambda_\beta$ does not seem to be possible easily.
\subsection{Related work}  \label{sec:related}
Let us briefly summarize some important known results on survival probabilities. For Brownian motion, the survival exponent is easily seen to be $\theta = 1/2$ by the reflection principle. The probability that a Brownian motion does not hit a moving boundary has also been studied in various cases. In this article, we will use some results of this type of \cite{uchiyama:1980}. \\
As mentioned in the introduction, for processes with independent and stationary increments, the problem can be solved using classical fluctuation theory. In particular, it has been shown that $\theta = 1/2$ for any random walk $S$ with centered increments and finite variance (see e.g.\ \cite{feller-vol2-1970}, Chapter XII, XIII and XVIII). In fact, the generating function of the first hitting time of the set $(0,\infty)$ can be computed explicitly in terms of the probabilities $\pr{}{S_n > 0}$ (Theorem XII.7.1 of \cite{feller-vol2-1970}). Similar results can be deduced for L\'evy processes, see e.g.\ \cite{doney} (p.\ 33); \cite{bertoin:1996}.\\ 
Apart from these facts, little is known outside the Gaussian framework. It has been shown that the survival exponent of integrated Brownian motion is $\theta = 1/4$ (\cite{ mckean:1963, goldman:1971, isozaki-watanabe:1994}). In fact, this is true for a much larger class of integrated L\'{e}vy processes and random walks, see  \cite{sinai:1992, aurzada-dereich:2009, vysotsky:2010,dembo-gao:2011}. For results on integrated stable L\'{e}vy processes, we refer to \cite{simon:2007}.\\
\cite{slepian:1962} studied survival probabilities for stationary Gaussian processes and obtained some general upper and lower bounds for their survival exponent. An important inequality (Slepian's inequality) is established that will be a very important tool throughout this work. It is also applied frequently in the work of \cite{li-shao:2004}, where universal upper and lower bounds for certain classes of Gaussian processes are derived. We further mention the works \cite{molchan:1999a} and \cite{molchan:1999}, where the survival exponent for fractional Brownian motion (FBM) is computed.
\section{The Gaussian case}\label{sec:gaussian}
Let $(X_n)_{n \geq 1}$ denote a sequence of independent standard normal random variables and let $B = (B_t)_{t \geq 0}$ denote a standard Brownian motion. For a measurable function $\sigma \colon [0,\infty) \to (0,\infty)$, let $Z$ be the corresponding weighted random walk defined in \eqref{def:weighted_RW}. Note that 
\begin{equation} \label{eq:eq_distr_gaussian_wrw}
   (Z_n)_{n \geq 1} \stackrel{d}{=} (B_{\kappa(n)})_{n \geq 1}, \qquad \kappa(n) := \sum_{k=1}^n \sigma(k)^2.
\end{equation}
The problem therefore amounts to determining the asymptotics of
\begin{equation} \label{eq:surv_prob_kappa}
   \pr{}{\sup_{n=1,\dots,N} B_{\kappa(n)} \leq 0}.
\end{equation}
Intuitively speaking, if $B_{\kappa(1)} \leq 0, \dots, B_{\kappa(N)} \leq 0$, then typically $B_{\kappa(N-1)}$ and $B_{\kappa(N)}$ are quite far away from the point $0$ if $N$ is large. One therefore expects that also $B_t \leq 0$ for $t \in [\kappa(N-1),\kappa(N)]$ unless the difference $\kappa(N) - \kappa(N-1)$ is so large that the Brownian motion has enough time to cross the $x$-axis with sufficiently high probability in the meantime. So if $\kappa(N) - \kappa(N-1)$ does not grow too fast, one would expect that the probability in \eqref{eq:surv_prob_kappa} behaves asymptotically just as in the case where the supremum is taken continuously over the corresponding interval (modulo terms of lower order). In the proof of Theorem~\ref{thm:polynomial_case_BM} and \ref{thm:subexp}, this idea will be made explicit in a slightly different way: we will require that the Brownian motion stays below a moving boundary on the intervals $[\kappa(N-1),\kappa(N)]$ where the moving boundary increases sufficiently slowly compared to $\kappa(N)$ in order to leave the survival exponent unchanged. We therefore split our results as follows: In Section \ref{sec:gauss_poly}, we consider polynomial functions $\kappa(N) = N^q$ for $q > 0$ (so $\kappa(N) - \kappa(N-1) \asymp N^{q-1}$). In Section \ref{sec:subexp}, we discuss the subexponential case $\kappa(N) = \exp(N^\alpha)$ for $0 < \alpha < 1$ (here  $\kappa(N) - \kappa(N-1) \asymp \kappa(N) N^{\alpha - 1}$) before finally turning to the exponential case $\kappa(N) = \exp(\beta N)$ for $\beta >0$ (now $ \kappa(N) - \kappa(N-1) \asymp \kappa(N)$) in Section \ref{sec:gaussian_exp}. \\
\begin{remark}\label{rem:value_of_barrier}
In the statement of Theorem~\ref{thm:polynomial_case_BM} and \ref{thm:subexp}, the value $0$ of the barrier can be replaced by any $c \in \mathbb{R}$ without changing the result. Indeed, let $\kappa \colon [0,\infty) \to (0,\infty)$ be such that $\kappa(N) \to \infty$ as $N \to \infty$ and let $a = \inf \set{\kappa(n) : n \in \mathbb{N}} > 0$. Note that for $c,d \in \mathbb{R}$, it holds that
\begin{align*}
 \pr{}{\sup_{n=1,\dots,N} B_{\kappa(n)} \leq c} &\geq \pr{}{B_{a/2} \leq c - d, \sup_{n=1,\dots,N} B_{\kappa(n)} - B_{a/2} \leq d} \\
&= \pr{}{B_{a/2} \leq c-d} \, \pr{}{\sup_{n=1,\dots,N} B_{\kappa(n) - a/2} \leq d} > 0.
\end{align*}
Now $\tilde{\kappa}(n) := \kappa(n) - a/2 > 0$ satisfies the same growth conditions as $\kappa$ stated in all theorems. Hence, it suffices to prove Theorem~\ref{thm:polynomial_case_BM} and \ref{thm:subexp} for the barrier $1$.
\end{remark}
\subsection{Polynomial case}\label{sec:gauss_poly}
The first result is a slightly more precise version of Theorem~\ref{thm:overview_1}.
\begin{theorem}\label{thm:polynomial_case_BM}
Let $\kappa \colon [0,\infty) \to (0,\infty)$ be a measurable function such that for some $q > 0$ and $\delta <q$
\begin{equation}
   \kappa(N) \asymp N^q \qquad \text{and } \qquad \kappa(N) - \kappa(N-1) \precsim N^{\delta}, \quad N \to \infty.
\end{equation}
Then for any $\gamma \in (\delta/2,q/2)$
\[
   N^{-q/2} \precsim \pr{}{\sup_{n=1,\dots,N} B_{\kappa(n)} \leq 0} \precsim N^{-q/2} (\log N)^{q/(4 \gamma - 2\delta)}, \quad N \to \infty.
\]
\end{theorem}
\begin{proof}
By assumption, there are constants $c_1,c_2 >0$ such that $c_1 n^q \leq \kappa(n) \leq c_2 n^q$ for $n$ large enough. The constant $c_2$ may be chosen so large that the second inequality holds for all $n \geq 1$. The lower bound is then easily established by comparison to the continuous time case if the barrier $0$ is replaced by $1$: 
\[
   \pr{}{\sup_{n=1,\dots,N} B_{\kappa(n)} \leq 1} \geq \pr{}{\sup_{t \in [0,c_2 N^q]} B_t \leq 1} \asymp N^{-q/2}.
\]
This also implies the same asymptotic order of the lower bound for any other barrier, see Remark~\ref{rem:value_of_barrier}.\\
For the proof of the upper bound, we will assume without loss of generality that $\kappa$ is nondecreasing. Otherwise, consider the continuous nondecreasing function $\tilde{\kappa}$ with $\tilde{\kappa}(n) = \max \set{ \kappa(l) : l = 0,\dots, n }$ for $n \in \mathbb{N}$ and $\tilde{\kappa}$ linear on $[n,n+1]$ for all $n \in \mathbb{N}$. Then $\tilde{\kappa}(N) \asymp N^q$ as $N \to \infty$. Moreover, $\tilde{\kappa}(N) - \tilde{\kappa}(N-1) = 0$ if $\kappa(N) \leq \tilde{\kappa}(N-1)$ and for $\kappa(N) > \tilde{\kappa}(N-1)$, we have
\[
  \tilde{\kappa}(N) - \tilde{\kappa}(N-1) = \kappa(N) - \tilde{\kappa}(N-1) \leq \kappa(N) - \kappa(N-1). 
\]
Thus, $\tilde{\kappa}$ satisfies the same growth conditions as $\kappa$. Clearly, for all $N$,
\[
  \pr{}{\sup_{n=1,\dots,N} B_{\kappa(n)} \leq 1} \leq \pr{}{\sup_{n=1,\dots,N} B_{\tilde{\kappa}(n)} \leq 1},
\]
so it suffices to prove the assertion of the theorem for a nondecreasing function $\kappa$. \\
Choose any $\gamma$ such that $\delta/2 < \gamma < q/2$ and set $g(N) = \lceil (K \cdot \log N)^{\frac{1}{2\gamma - \delta}} \rceil$ for some $K$ (which has to be chosen large enough later on). Next, note that
\begin{align*}
	\bigcap_{n=g(N)}^N \set{ B_{\kappa(n)} \leq 1 } &\subseteq \bigcap_{n=g(N)}^{N-1} \set{ \sup_{t \in [\kappa(n),\kappa(n+1)]} B_t \leq n^\gamma + 1 } \\
	&\quad \cup \bigcup_{n=g(N)}^{N-1} \set{ \sup_{t \in [\kappa(n),\kappa(n+1)]} B_t - B_{\kappa(n)} > n^\gamma } =: G_N \cup H_N.
\end{align*}
Clearly, it holds that
\[
	\pr{}{\sup_{n=1,\dots,N} B_{\kappa(n)} \leq 1} \leq \pr{}{\sup_{n= g(N),\dots,N} B_{\kappa(n)} \leq 1} \leq \pr{}{G_N} + \pr{}{H_N}.
\]
Next, note that $\kappa(n) \leq t$ implies that $n \leq (t/c_1)^{1/q}$ for $n$ sufficiently large. Using also that $\kappa(\cdot)$ is nondecreasing, we obtain that
\begin{align*}
	\pr{}{G_N} &\leq \pr{}{
	\bigcap_{n=g(N)}^{N-1} \set{ \sup_{t \in [\kappa(n),\kappa(n+1)]} B_t - (t/c_1)^{\gamma/q} \leq 1 } } \\
	&= \pr{}{
 	\sup_{t \in [\kappa(g(N)),\kappa(N)]} B_t - (t/c_1)^{\gamma/q} \leq 1 } =: p_1(N).	
\end{align*}
Moreover, using the stationarity of increments and the scaling property of Brownian motion, we obtain the following estimates:
\begin{align*}
	\pr{}{H_N} &\leq \sum_{n=g(N)}^{N-1} \pr{}{\sup_{t \in [0,\kappa(n+1) - \kappa(n)]} B_t > n^\gamma} \\
 	&=\sum_{n=g(N)}^{N-1} \pr{}{\sup_{t \in [0,1]} B_t > \frac{n^\gamma}{\sqrt{\kappa(n+1) - \kappa(n)}}} =:  p_2(N).	
\end{align*}
 Let us first show that the second term $p_2$ decays faster than $N^{-q/2}$ as $N \to \infty$. To this end, let $c$ denote a constant such that $\kappa(n+1) - \kappa(n) \leq c \, n^{\delta}$ for all $n$ sufficiently large. In particular, for $N$ large enough,
\begin{align*}
p_2(N) &\leq N \, \max_{n=g(N),\dots,N}  \pr{}{\sup_{t \in [0,1]} B_t > c^{-1/2} n^{\gamma - \delta/2} }\\
&= N \, \pr{}{\sup_{t \in [0,1]} B_t > c^{-1/2} g(N)^{\gamma - \delta/2} },
\end{align*}
 since $\gamma$ was chosen such that $\gamma - \delta/2 > 0$. Next, recalling that 
\begin{equation}\label{eq:tails_of_BM}
 \pr{}{\sup_{t \in [0,1]} B_t > u} = \pr{}{\abs{B_1} >u} = \sqrt{\frac{2}{\pi}} \, \int_u^\infty {e^{-x^2/2}} \,dx \leq e^{-u^2/2}, \quad u \geq 0,
\end{equation}
we may finally conclude that
\begin{align*}
p_2(N) \leq N \exp\left(-\frac{g(N)^{2\gamma - \delta}}{2 c}\right) \leq N^{1 - \frac{K}{2 \, c}}. 
\end{align*}
By choosing $K$ large enough, the assertion that $p_2$ decreases faster than $N^{-q/2}$ is verified.\\
It remains to show that $p_1(N) \precsim \,N^{-q/2} \, (\log N)^{q/(4\gamma - 2 \delta)}$. Let
\[
 F(t) :=
	\begin{cases}
 		c_1^{-\gamma/q} t^{\gamma / q}, &\quad t \geq c_1,\\
		1, &\quad t < c_1.
	\end{cases}
\]
Clearly we have for $N$ large enough that
\begin{align*}
 p_1(N) &= \pr{}{\sup_{t \in  [\kappa(g(N)),\kappa(N)]} B_t - c_1^{-\gamma/q}t^{\gamma/q} \leq 1} \\
&= \pr{}{\sup_{t \in  [\kappa(g(N)),\kappa(N)]} B_t - F(t) \leq 1}.
\end{align*}
Since $\E{}{B_s B_t} \geq 0$ for all $s,t \geq 0$, Slepian's inequality (cf.\ Theorem~3 of \cite{slepian:1962}) implies that
\begin{align*}	
	&\pr{}{\sup_{t \in  [\kappa(g(N)),\kappa(N)]}B_t - F(t) \leq 1} \leq 
	\frac{\pr{}{\sup_{t \in  [0,\kappa(N)]} B_t - F(t) \leq 1}}{\pr{}{\sup_{t \in  [0,\kappa(g(N))]} B_t - F(t) \leq 1}}.
\end{align*}
One has to determine the probability that a Brownian motion does not hit the moving boundary $1+F(\cdot)$. Now 
\[
   \pr{}{B_t \leq c \, t^\alpha + 1, \, \forall t \in [0,T]} \asymp \pr{}{\sup_{0 \leq t \leq T} B_t \leq 1} \asymp T^{-1/2}, \quad T \to \infty
\]
if $\alpha < 1/2$ and $c>0$ by Theorem~5.1 of \cite{uchiyama:1980}, i.e.\ adding a drift of order $t^\alpha$ ($\alpha < 1/2$) to a Brownian motion does not change the rate $T^{-1/2}$. Since $\gamma/q <1/2$, this implies for the boundary $1 + F(\cdot)$ that
\begin{align*}   
&\frac{\pr{}{\sup_{t \in  [0,\kappa(N)]} B_t - F(t) \leq 1}}{\pr{}{\sup_{t \in  [0,\kappa(g(N))]} B_t - F(t) \leq 1}} \asymp \frac{\pr{}{\sup_{t \in  [0,\kappa(N)]}B_t \leq 1}}{\pr{}{ \sup_{t \in  [0,\kappa(g(N))]} B_t \leq 1}}\\
&\quad \asymp  \kappa(g(N))^{1/2} \, \kappa(N)^{-1/2} \asymp  (\log N)^{q/(4\gamma - 2\delta)} \, N^{-q/2}.
\end{align*}
\end{proof}
\begin{remark}
   The assertion of the proposition above becomes false if we remove the condition $\kappa(N+1) - \kappa(N) \precsim N^\delta$ for some $\delta < q$. Indeed, let $q>0$ and for $n \in \mathbb{N}$, set $\kappa(n) = \exp(q k)$ if $e^k \leq n < e^{k+1}$ for $k \in \mathbb{N}$. Then $\kappa(N) \asymp N^q$ as $N \to \infty$. Moreover, $\kappa(N+1) - \kappa(N) = 0$ if there is $k \in \mathbb{N}$ such that $N, N+1 \in [e^k,e^{k+1})$ and 
\[
   \kappa(N+1) - \kappa(N) = \exp(q (k+1)) - \exp(q k) = \kappa(N) (e^q - 1) 
\]
for $k \in \mathbb{N}$ such that $e^k \leq N < e^{k+1} \leq N+1$. In particular, $\kappa(N+1) - \kappa(N) \precsim N^q$. Next, note that 
\begin{align*}
   & \pr{}{\sup_{n=1,\dots,N} B_{\kappa(n)} \leq 0} = \pr{}{B(e^{qk}) \leq 0, k \in \mathbb{N}, e^k \leq N} \\
&\quad= \pr{}{\bigcap_{k=1}^{\lfloor \log N \rfloor} \set{B(e^{qk}) \leq 0}} \geq \prod_{k=1}^{\lfloor \log N \rfloor} \pr{}{B(e^{qk}) \leq 0} \geq (1/2)^{\log N} = N^{-\log 2}. 
\end{align*}
The first inequality holds by Slepian's inequality (see also \eqref{eq:one_half_to_N}). Hence, $N^{-q/2}$ cannot be an upper bound for the survival probability if $q > 2 \log 2$.
\end{remark}
\subsection{Subexponential case}\label{sec:subexp}
Here we consider functions $\kappa(\cdot)$ that grow faster than any polynomial but slower than any exponential function, i.e.
\[
   \lim_{N \to \infty} \frac{N^q}{\kappa(N)} = 0, \quad q >0, \qquad \lim_{N \to \infty} \frac{\kappa(N)}{e^{\beta N}} = 0, \quad \beta >0.
\]
For simplicity, we restrict our attention to the natural choice $\kappa(n) \asymp \exp(\nu \, n^\alpha)$ for $\nu>0, \alpha \in (0,1)$. Under certain additional assumptions the proof of Theorem~\ref{thm:polynomial_case_BM} can be adapted to yield the following result:
\begin{theorem}\label{thm:subexp}
Let $\kappa \colon [0,\infty) \to (0,\infty)$ be a measurable function such that 
\[
    \kappa(N) \asymp \exp(\nu \, N^\alpha), \qquad \kappa(N+1) - \kappa(N) \precsim \kappa(N) \, N^{-\gamma}, \quad N \to \infty,
\]
where $\alpha, \nu > 0$ and $\gamma > 3 \alpha$. Then
\[
   \lim_{N \to \infty} N^{-\alpha} \, \log \pr{}{\sup_{n=1,\dots,N} B_{\kappa(n)} \leq 0} = -\nu/2.
\]
More precisely, for $\Lambda := \alpha/(\gamma - 2\alpha) < 1$, one has
\[
   \exp \left( - \frac{\nu}{2} \, N^{\alpha} \right) \precsim \pr{}{\sup_{n=1,\dots,N} B_{\kappa(n)} \leq 0} \precsim \exp \left( - \frac{\nu}{2} \, N^{\alpha} \right) \cdot \exp \left( N^{\Lambda \alpha + o(1)}  \right).
\]
\end{theorem}
\begin{proof}
For simplicity of notation, we again use the barrier $1$ instead of $0$. The result then follows in view of Remark~\ref{rem:value_of_barrier}. By assumption, there are constants $N_0, c_1,c_2>0$ such that $c_1 \exp(\nu \, n^\alpha) \leq \kappa(n) \leq c_2 \exp(\nu \, n^\alpha)$ for all $n \geq N_0$. So for all $N \geq N_0$, we have
\[
  \pr{}{\sup_{n=1,\dots,N} B_{\kappa(n)} \leq 1} \geq \pr{}{\sup_{t \in [0,c_2 \exp(\nu \, n^\alpha)]} B_t \leq 1 } \asymp \exp \left( -\frac{\nu}{2} \, N^\alpha \right), \quad N \to \infty.
\]
For the proof of the upper bound, we assume w.l.o.g.\ that $\kappa$ is nondecreasing (see the proof of Theorem~\ref{thm:polynomial_case_BM}). The assumption $\gamma > 3 \alpha$ allows us to find a constant $\rho$ with $\alpha < \rho < \gamma/2$ and  $\delta := \alpha/(\gamma - 2 \rho) < 1$. Set 
\[
   f(t) :=  \exp\left(\frac{\nu}{2} \, t^\alpha \right) \,t^{-\rho}, \qquad g(t) := \lceil t^\delta \rceil, \quad t > 0.
\]
As in the proof of Theorem~\ref{thm:polynomial_case_BM}, it holds that
\[
	\pr{}{\sup_{n=1,\dots,N} B_{\kappa(n)} \leq 1} \leq \pr{}{\sup_{n= g(N),\dots,N} B_{\kappa(n)} \leq 1} \leq \pr{}{G_N} + \pr{}{H_N},
\]
where
\begin{align*}
 G_N &:= \bigcap_{n=g(N)}^{N-1} \set{ \sup_{t \in [\kappa(n),\kappa(n+1)]} B_t \leq f(n) + 1 },\\
 H_N &:= \bigcup_{n=g(N)}^{N-1} \set{ \sup_{t \in [\kappa(n),\kappa(n+1)]} B_t - B_{\kappa(n)} > f(n) }.
\end{align*}
Note that $t \geq \kappa(n)$ implies that $n \leq (\log(t/c_1)/\nu)^{1/\alpha} =: h(t)$ if $n \geq N_0$. Keeping in mind that $f(\cdot)$ is ultimately increasing, we obtain the following estimates for large $N$ and some constant $c_3 >0$:
\begin{align*}
 \pr{}{G_N} &\leq \pr{}{ \bigcap_{n=g(N)}^{N-1} \set{ \sup_{t \in [\kappa(n),\kappa(n+1)]} B_t - f\left(h(t)\right) \leq 1 } } \\
 &= \pr{}{ \sup_{t \in [\kappa(g(N)),\kappa(N)]} B_t - c_3 \frac{\sqrt{t}}{(\log t)^{\rho / \alpha} }  \leq 1 } =: p_1(N).
\end{align*}
Next, using the stationarity and the scaling property of Brownian motion, we have that
\[
	\pr{}{H_N} \leq \sum_{n=g(N)}^{N-1} \pr{}{\sup_{t \in [0,1]} B_t > \frac{f(n)}{\sqrt{\kappa(n+1) - \kappa(n)}}} =: p_2(N).
\]
We first show that the term $p_2$ is of lower order than $\exp(-N^\alpha)$. To this end, since $\kappa(N+1) - \kappa(N) \leq c_4 \kappa(N) N^{-\gamma} \leq c_5 \exp(\nu N^\alpha) \, N^{-\gamma}$ for all $N$ sufficiently large and some constants $c_4,c_5 > 0$, we get 
\begin{align*}
p_2(N) &\leq N \, \max_{n=g(N),\dots,N}  \pr{}{\sup_{t \in [0,1]} B_t > c_5^{-1/2} \, n^{\gamma/2 - \rho}  } \\
&= N \, \pr{}{\sup_{t \in [0,1]} B_t > c_5^{-1/2} \, g(N)^{\gamma/2 - \rho} },
\end{align*}
since $\gamma/2 - \rho > 0$ by the choice of $\rho$. Recalling \eqref{eq:tails_of_BM}, we obtain
\begin{align*}
p_2(N) \leq N \exp \left( -\frac{1}{2c_5} \, g(N)^{\gamma - 2 \rho} \right) \precsim N \exp \left( -\frac{1}{2c_5} \, N^{\delta(\gamma - 2 \rho)} \right) , \quad N \to \infty.
\end{align*}
Now $\delta(\gamma - 2 \rho) > \alpha$ by the choice of $\delta$, so this term is $o(\exp(-N^\alpha))$. \\
It remains to show that 
\[
   p_1(N) \precsim \exp \left( - \frac{\nu}{2} N^{\alpha} \right) \cdot \exp \left( \frac{\nu}{2} N^{\delta \alpha}  \right), \quad N \to \infty.
\]
 Set
\[
 F(t) :=
	\begin{cases}
 		c_3 \, \frac{\sqrt{t}}{(\log t)^{\rho / \alpha} }, &\quad t \geq d_1,\\
		d_2, &\quad t < d_1.
	\end{cases}
\]
where $d_1,d_2$ are chosen in such a way that $F$ is nondecreasing and continuous. By Slepian's inequality, one has for $N$ sufficiently large
\begin{align*}	
	p_1(N) = \pr{}{\sup_{t \in  [\kappa(g(N)),\kappa(N)]}B_t - F(t) \leq 1} \leq 
	\frac{\pr{}{\sup_{t \in  [0,\kappa(N)]} B_t - F(t) \leq 1}}{\pr{}{\sup_{t \in  [0,\kappa(g(N))]} B_t - F(t) \leq 1}}.
\end{align*}
Theorem~5.1 of \cite{uchiyama:1980} ensures that the drift $F(\cdot)$ does not change the rate of the survival probability since for some $d_3 > 0$, we have
\[
   \int_1^\infty F(t) t^{-3/2} \,dt = d_3 + \int_{d_1}^\infty \frac{1}{t (\log t)^{\rho / \alpha}} < \infty
\]
because $\rho > \alpha$, and therefore, 
\begin{align*}   
&\frac{\pr{}{\sup_{t \in  [0,\kappa(N)]} B_t - F(t) \leq 1}}{\pr{}{\sup_{t \in  [0,\kappa(g(N))]} B_t - F(t) \leq 1}} \asymp \frac{\pr{}{\sup_{t \in  [0,\kappa(N)]} B_t \leq 1}}{\pr{}{\sup_{t \in  [0,\kappa(g(N))]} B_t  \leq 1}} \\
&\quad \asymp \kappa(g(N))^{1/2} \, \kappa(N)^{-1/2} \asymp \exp \left( - \frac{\nu}{2} \, N^{\alpha} \right) \cdot \exp \left( \frac{\nu}{2} \, N^{\delta \alpha}  \right),  \quad N \to \infty.
\end{align*}
Finally $\delta = \alpha/(\gamma - 2 \alpha) + o(1) = \Lambda + o(1)$ as $\rho \downarrow \alpha$.
\end{proof}

\begin{corollary}
   If $\kappa(n) = \exp(\nu \, n^\alpha)$ for some $\nu >0$ and $\alpha \in (0,1/4)$, then
\[
   \lim_{N \to \infty} N^{-\alpha} \, \log \pr{}{\sup_{n=1,\dots,N} B_{\kappa(n)} \leq 0} = -\nu/2.
\]
\end{corollary}
\begin{proof}
Note that
    \begin{align*}
    & \kappa(N+1) - \kappa(N) = \kappa(N) (e^{\nu ((N+1)^\alpha - N^{\alpha})} - 1) \sim \nu \kappa(N) ((N+1)^\alpha - N^{\alpha}) \\
 &= \nu \kappa(N) N^{\alpha-1} \frac{(1 + 1/N)^\alpha - 1}{1/N} \sim \alpha \nu \, \kappa(N) N^{\alpha-1}
 \end{align*}
Hence, we can apply Theorem~\ref{thm:subexp} with $\gamma = 1 - \alpha$ if $\gamma > 3 \alpha$, i.e.\ for $\alpha \in (0,1/4)$.
\end{proof}
\begin{remark}
The case $\alpha \geq 1/4$ remains unsolved. In view of the heuristics presented below \eqref{eq:surv_prob_kappa}, it would be interesting to know whether 
\[
  \liminf_{N \to \infty} \frac{\log \pr{}{\sup_{n=1,\dots,N} B_{\kappa(n)} \leq 1}}{N^\alpha} > -\nu/2 = \lim_{N \to \infty} \frac{\log \pr{}{\sup_{t \in [1,N]} B_{\kappa(t)} \leq 1}}{N^\alpha}
\]
for some $\alpha \in [1/4,1)$. At least for $\alpha = 1$, the rate of decay of the continuous time and discrete time survival probability is different in general as we prove in the next subsection, cf.\ \eqref{eq:lambda_upper_bound}.
\end{remark}
\subsection{Exponential case}\label{sec:gaussian_exp}
In this section, we consider the asymptotic behaviour of 
\[
   \pr{}{\sup_{n=0,\dots,N} B(e^{\beta n}) \leq 0}
\]
where $\beta > 0$. It will be helpful to rewrite the process as a discrete Ornstein-Uhlenbeck process. Indeed, observe that 
\[
   \pr{}{\sup_{n=0,\dots,N}B(e^{\beta n}) \leq 0} = \pr{}{\sup_{n=0,\dots,N} e^{-\beta n /2} B(e^{\beta n}) \leq 0} = \pr{}{\sup_{n=0,\dots,N} U_{\beta n} \leq 0}
\]
where $(U_t)_{t \geq 0}$ is the Ornstein-Uhlenbeck process, i.e.\ a centered stationary Gaussian process with covariance function 
\[
   \rho(t,s) = \E{}{U_t U_s} = e^{-\abs{t-s}/2}.
\]
To our knowlegde, the survival probability of the discrete Ornstein-Uhlenbeck process has not been computed in the literature. For the continuous time case, it is has been shown that
\begin{equation}\label{eq:OU_sup_distr}
   \pr{}{\sup_{t \in [0,T]} U_t \leq 0} = \frac{1}{\pi} \, \arcsin(e^{-T/2}),
\end{equation}
see e.g.\ \cite{slepian:1962}. In fact, this relation can be established by direct computation using an intergral formula (see Eq. 6.285.1 of \cite{gradshteyn-ryzhik:2000}). It is important to remark that the survival exponent of the Ornstein-Uhlenbeck process does depend on the value of the barrier, i.e.\ for $c>0$
\[
 \pr{}{\sup_{t \in [0,T]} U_t \leq c} \asymp \exp(-\theta(c) \, T), \quad T \to \infty,
\]
for some decreasing function $\theta \colon [0,\infty) \to (0,1/2]$. We refer to \cite{beekman:1977,sato:1977} for more details and related results. In the sequel, we work with the barrier $c=0$ although the techniques presented are applicable for $c \neq 0$ as well.\\
If $p(n) = \pr{}{U_0 \leq 0, U_{\beta} \leq 0 \dots, U_{\beta n} \leq 0}$, Slepian's inequality and the stationarity of $U$ imply that $p(n+m) \geq p(n) \, p(m)$. By the usual subadditivity argument this implies the existence of $\lambda_\beta \in (0,\infty]$ such that
\begin{equation}\label{eq:exp_rate_gaussian}
   \lim_{N \to \infty} -\frac{1}{N} \log \pr{}{\sup_{n=0,\dots,N} U_{\beta n} \leq 0} = \lambda_\beta.
\end{equation}
Slepian's inequality further implies that $\beta \mapsto \lambda_\beta$ is nondecreasing. 
Unfortunately, we are not able to obtain an explicit expression for $\lambda_\beta$. However, we provide several estimates which are summarized in 
\begin{theorem}\label{thm:summary_exp}
For all $\beta >0$, we have that
\begin{equation}\label{eq:lambda_lower_bound}
   \lambda_\beta \geq 
\begin{cases}
   \log(2) - c(\beta), &\quad \beta > \beta_0,\\
 (\log(2) - c(\beta m))/m, &\quad \beta \in (0,\beta_0], \quad m = \lceil \beta_0 / \beta \rceil,
\end{cases}
\end{equation}
where
\[
   c(x) := \frac{e^{-x/2}}{1 - e^{-x/2}}, \quad x > 0, \quad \beta_0 := 2 \log(1 + 1/\log 2) \approx 1.786.
\]
Moreover,
\begin{equation}\label{eq:lambda_upper_bound}
   \lambda_\beta \leq \begin{cases}
   \beta/2, &\quad \beta \in (0,\beta_1],\\
 \log(2) - \log\left(1 + \frac{2}{\pi}  \arcsin\left(e^{-\beta/2}\right) \right) , &\quad \beta \in [\beta_1,\infty),
\end{cases}
\end{equation}
where $\beta_1 \approx 0.472$ is the unique solution on $(0,\infty)$ to the equation 
\[
 \frac{\beta}{2} = \log(2) - \log\left(1 + \frac{2}{\pi}  \arcsin\left(e^{-\beta/2}\right) \right).
\]
\end{theorem}
\begin{remark}
   For $\beta > \beta_0$, the above theorem implies that
\[
   \frac{2}{\pi} e^{-\beta/2} \sim \log\left(1 + \frac{2}{\pi}  \arcsin\left(e^{-\beta/2}\right) \right) \leq \log(2) - \lambda_\beta \leq c(\beta) \sim e^{-\beta/2}, \quad \beta \to \infty,
\]
i.e.\ $\lambda_\beta \downarrow \log 2$ exponentially fast as $\beta \uparrow \infty$. \\
However, it remains an open question whether $\lambda_\beta$ is stricly less than $\beta/2$ also for $\beta < \beta_1$ (this would imply that the rate in the discrete time and continuous time framework does not coincide for all $\beta$) and whether $\lambda_\beta \sim \beta/2$ as $\beta \downarrow 0$. 
\end{remark}
\subsubsection{Upper bounds for the survival probability}
Here we prove the first part of the inequality \eqref{eq:lambda_lower_bound}.
\begin{lemma}\label{lem:upper_bound_exp_BM}
Let $\beta > \beta_0 = 2 \log(1 + 1/\log 2)$. Then for all $N$
\[
   \pr{}{\sup_{n=0,\dots,N} B(e^{\beta n}) \leq 0 } \leq \frac{1}{2}\exp\left( -\left(\log 2 - c(\beta)\right) \, N  \right).
\]
where $c(\beta) \in (0,\log 2)$ is defined in Theorem~\ref{thm:summary_exp}.
\end{lemma}
\begin{proof}
First, note that $c(\cdot)$ is decreasing with $c(\beta_0) = \log 2$. Since $\set{ B(e^{\beta n}) \leq 0 }  = \set{ U_{\beta n} \leq 0}$, we have by Corollary 2.3 of \cite{li-shao:2002} 
\begin{align*}
   \pr{}{\sup_{n=0,\dots,N} B(e^{\beta n}) \leq 0} &\leq \prod_{n=1}^{N+1} \pr{}{U_{\beta (n-1)} \leq 0} \, \exp\left(\sum_{1 \leq i < j \leq N + 1} e^{-\beta \abs{i - j}/2} \right) \\
	&= 2^{-(N+1)} \, \exp\left(\sum_{1 \leq i < j \leq N + 1} e^{-\beta \abs{i - j}/2} \right).
\end{align*}
One computes
\begin{align*}
 \sum_{1 \leq i < j \leq N+1} e^{-\beta \abs{i - j}/2} &= \sum_{i=1}^{N} \sum_{j=i+1}^{N+1} e^{-\beta (j-i)/2} = \sum_{i=1}^{N} \sum_{j=1}^{N+1-i} e^{-\beta j/2}\\
 &= c(\beta) \sum_{i=1}^{N} (1 - e^{-\beta(N+1-i)/2}) \leq c(\beta) \, N.
\end{align*}
\end{proof}
Next, we prove the second part of \eqref{eq:lambda_lower_bound}. For small $\beta$, we rescale the exponent of the weight function in order to apply Lemma \ref{lem:upper_bound_exp_BM}.
\begin{lemma}
Let $0 < \beta < \beta_0$ and set $m = m_\beta = \lceil \beta_0 / \beta \rceil$. Then
\[
   \pr{}{\sup_{n=0,\dots,N} B(e^{\beta n}) \leq 0 } \leq \exp\left( -\frac{\log 2 - c(\beta m)}{m} \, N - c(\beta m)\right), \quad N > m.
\]
\end{lemma}
\begin{proof}
   Clearly, for $N > m$, 
\begin{align*}
	&\pr{}{\sup_{n=0,\dots,N} B(e^{\beta n}) \leq 0 } \leq \pr{}{B_1 \leq 0, \, \sup_{n=m,\dots,N} B(e^{\beta n}) \leq 0 } \\
&\quad = \pr{}{B_1 \leq 0, \, \sup_{n \in \set{1,(m+1)/m,\dots,(N-1)/m,N/m}} B(e^{m\beta n}) \leq 0} \\ 
&\quad \leq \pr{}{\sup_{n \in \set{0,1,2,\dots,\lfloor N/m \rfloor}} B(e^{m\beta n}) \leq 0} \leq e^{-(\log 2 - c(\beta m)) \lfloor N/m \rfloor} /2  
\end{align*}
by Lemma \ref{lem:upper_bound_exp_BM} since  $\beta m > \beta_0$. Using that $\lfloor N/m \rfloor \geq N/m - 1$, the assertion follows.
\end{proof}
\subsubsection{Lower bounds for the survival probability}
We now prove \eqref{eq:lambda_upper_bound}. In view of \eqref{eq:OU_sup_distr}, a comparison to the continuous time framework yields
\[
   \pr{}{\sup_{n=0,\dots,N} B(e^{\beta n}) \leq 0} \geq  \pr{}{\sup_{0 \leq t \leq N} U_{\beta t} \leq 0} \sim \pi^{-1} \, e^{-\beta N/2}, \quad N \to \infty. 
\]
Obviously, for any sequence $0 = t_0 < t_1 < \dots < t_N$, we have
\begin{equation}\label{eq:one_half_to_N}
   \pr{}{\sup_{n=1,\dots,N} B(t_n) \leq 0} \geq \pr{}{B(t_1) \leq 0, \sup_{n=2,\dots,N} B(t_n) - B(t_{n-1}) \leq 0} = 2^{-N},
\end{equation}
by independence and symmetry of the increments (or simply Slepian's inequality again). \\
For the exponential case, simple lower bounds are therefore
\[
   \pr{}{ \sup_{n=0,\dots,N}B(e^{\beta n}) \leq 0} \succsim \exp(- (\tfrac{\beta}{2} \wedge \log 2) \cdot N), \qquad N \to \infty.
\]
In particular, this shows that $\lambda_\beta \leq \beta/2$ as stated in \eqref{eq:lambda_upper_bound}. The fact that the probability $\pr{}{B_t \leq 0, B_s \leq 0}$ admits an explicit formula in terms of $s$ and $t$ can be used to establish a new lower bound that improves the trivial bound $\log 2$ and completes the proof of \eqref{eq:lambda_upper_bound}.
\begin{lemma}\label{lem:lower_bound_exp_BM}
    \[
        \pr{}{\sup_{n=0,\dots,N} B(e^{\beta n}) \leq 0} \geq \frac{1}{2} \left( \frac{1}{2} + \frac{1}{\pi} \arcsin \left( e^{-\beta/2} \right)  \right)^N. 
    \]
\end{lemma}
\begin{proof}
   Let $A_n := \set{\sup_{k = 0, \dots, n} B(e^{\beta k}) \leq 0}$. Then
 \begin{align*}
     \pr{}{A_N} &= \pr{}{B(e^{\beta N}) \leq 0 \vert A_{N-1}} \, \pr{}{A_{N-1}} 
 = \pr{}{X_0 \leq 0} \prod_{n=1}^N \pr{}{B(e^{\beta n}) \leq 0 \vert A_{n-1}} \\
& \geq \frac{1}{2} \prod_{n=1}^N \pr{}{B(e^{\beta n}) \leq 0 \vert B(e^{\beta (n-1)}) \leq 0 },
 \end{align*}
where the inequality follows from Lemma 5 of \cite{bramson:1978}. Next, recall that
\[
   \pr{}{B_s \leq 0, B_t \leq 0} = \frac{1}{4} + \frac{1}{2\pi} \arctan \left( \sqrt{\frac{s}{t-s}} \right), \quad s < t,
\]
see e.g.\ Exercise 8.5.1 in \cite{grimmett-stirzaker:2001}. In particular,
\[
   \pr{}{B(e^{\beta n}) \leq 0 \vert B(e^{\beta (n-1)}) \leq 0 } = \frac{1}{2} + \frac{1}{\pi} \arctan \left( \frac{1}{\sqrt{\exp(\beta)-1}} \right), \quad n \geq 1,
\]
independent of $n$. Now use that $\arctan(x) = \arcsin(x/\sqrt{x^2 + 1})$.
\end{proof}
\subsubsection{A related Fredholm integral equation}
If $(Y_n)_{n \geq 0}$ is a sequence of independent standard normal random variables, set 
\[
   X_0 = Y_0, \quad X_n = e^{-\beta/2} X_{n-1} + (1 - e^{-\beta})^{1/2} Y_n, \quad n \geq 1.
\]
One can check that $(X_n)_{n \geq 0}$ and $(U(\beta n))_{n \geq 0}$ are equal in distribution. The above recursion equation is a special case of an autogregressive model of order $1$ (AR(1)-model) that can also be used to define a discrete version of the Ornstein-Uhlenbeck process if the $Y_n$ are not necessarily Gaussian, see e.g.\ \cite{larralde:2004}. Larralde explicitly computes the generating function of the first hitting time of the set $(0,\infty)$ if the $Y_n$ have a two-sided exponential distribution. Conditions ensuring that exponential moments of the first hitting time of the set $[x,\infty)$ ($x \geq 0$) exist for an AR(1) process can be found in \cite{novikov-kordzakhia:2008}.
\\
We only discuss the case of standard normal random variables $Y_n$. Recall from the beginning of Section~\ref{sec:gaussian_exp} that $(X_n)_{n \geq 0}$ is a stationary Markov chain with transition density 
\[
	p(x,y) := \frac{1}{\sqrt{2\pi}\sigma} \exp \left( -\frac{(y-\rho x)^2}{2 \sigma^2} \right), \quad x,y \in \mathbb{R},
\] 
where $\rho = e^{-\beta/2}$ and $\sigma = \sqrt{1-e^{-\beta}}$. Set $A_n := \set{X_0 \leq 0, \dots, X_n \leq 0}$ and let $\pi_n$ be the law of $X_n$ given $A_n$, i.e.
\[
   \pi_n((-\infty,u]) := \pr{}{X_n \leq u \vert A_n}, \quad u \leq 0.
\]
\begin{proposition}\label{prop:int_eq}
It holds that 
\[
   \lim_{N \to \infty} -\frac{1}{N} \, \log \pr{}{\sup_{n=0,\dots,N} X_n \leq 0} = \lambda_\beta
\]
where 
\[
   \pr{}{X_n \leq 0 \vert A_{n-1}} \nearrow \exp(-\lambda_\beta), \quad n \to \infty.
\]
Moreover, the sequence $(\pi_n)_{n \geq 0}$ converges weakly to a probability measure $\pi$ on $(-\infty,0]$ which is absolutely continuous w.r.t.\ the Lebesgue measure on $(-\infty,0]$. Denote its density by $\varphi$. Then $\varphi$ satisfies the following Fredholm integral equation of second kind:
\[
   \exp(-\lambda_\beta) \, \varphi(u) = \int_{-\infty}^0 p(y,u) \, \varphi(y) \, dy, \quad u \leq 0.
\]
\end{proposition}
\begin{proof}
Let $F_n(u) := \pr{}{X_n \leq u \vert A_{n-1}}$, $u \leq 0$. Note that for $u \leq 0$
\begin{equation}\label{eq:conv_mu_proof}
   \pi_n((-\infty,u]) = \frac{\pr{}{X_n \leq u, A_{n-1}}}{\pr{}{A_n}} = \frac{\pr{}{X_n \leq u \vert A_{n-1}}}{\pr{}{ X_n \leq 0 \vert A_{n-1}}} = \frac{F_n(u)}{F_n(0)}.
\end{equation}
Moreover, for $u \leq 0$, we have
\begin{align*}
   F_n(u) &= \pr{}{X_n \leq u \vert A_{n-1}} = \int_{-\infty}^0 \pr{}{X_n \leq u \vert X_{n-1} = y} \pr{}{X_{n-1} \in dy \vert A_{n-1}} \\
&= \int_{-\infty}^0 \int_{-\infty}^u p(y,z) \, dz \, \pi_{n-1}(dy).
\end{align*}
Assume for a moment that $F_n(u)$ converges to $F(u)$ for all $u \leq 0$ and that $(\pi_n)_{n \geq 1}$ converges weakly to some probability measure $\pi$. Then the last equation and \eqref{eq:conv_mu_proof} imply that
\[
   \pi((-\infty,u]) = \frac{F(u)}{F(0)} = \frac{1}{F(0)} \, \int_{-\infty}^0 \int_{-\infty}^u p(y,z) \, dz \, \pi(dy), \quad u \leq 0.
\]
Applying Fubini's theorem, the previous equation reads
\begin{equation}\label{eq:stat_measure_proof}
   F(0) \, \pi((-\infty,u]) = \int_{-\infty}^u \int_{-\infty}^0 p(y,z) \, \pi(dy) \, dz, \quad u \leq 0.
\end{equation}
One can then conclude that $\pi$ is absolutely continuous w.r.t.\ to the Lebesgue measure. Denote its density by $\varphi$. Differentiating \eqref{eq:stat_measure_proof} w.r.t.\ to $u$, we conclude that
\begin{equation}\label{eq:stat_density_proof}
  F(0) \, \varphi(u) = \int_{-\infty}^0 p(y,u) \, \varphi(y) \,dy \quad \text{for all } u < 0.
\end{equation}
In order to prove convergence of $F_n(u)$ for $u \leq 0$, it suffices to show that $F_n(u)$ is non-decreasing in $n$. Indeed,
\begin{align*}
   F_{n+1}(u) &= \pr{}{X_{n+1} \leq u \vert X_0 \leq 0, \dots, X_n \leq 0} \\
&\geq \pr{}{X_{n+1} \leq u \vert X_1 \leq 0, \dots, X_n \leq 0} = F_n(u). 
\end{align*}
The inequality follows from Lemma 5 of \cite{bramson:1978}, the last equality is due to the stationarity of $X$. Using \eqref{eq:conv_mu_proof}, it is not hard to show that the sequence $(\pi_n)_{n \geq 0}$ converges weakly to some probability measure $\pi$. Next, since
\[
   F_n(0) = F(0) \, (1 + g(n)), \quad n \geq 0,
\]
where $g(n) \to 0$ as $n \to \infty$, we get
\begin{align*}
   \pr{}{\sup_{n=0,\dots,N} X_n \leq 0} &= \pr{}{X_N \leq 0 \vert A_{N-1}} \, \pr{}{A_{N-1}} \\
&= \pr{}{X_0 \leq 0} \prod_{n=1}^N \pr{}{X_n \leq 0 \vert A_{n-1}} \\
&= \frac{1}{2} F(0)^N \, \exp\left( \sum_{n=1}^N \log(1 + g(n)) \right) = F(0)^N \, e^{o(N)}.
\end{align*}
One concludes (recall \eqref{eq:exp_rate_gaussian}) that
\[
 \lim_{N \to \infty} -\frac{1}{N} \, \log \pr{}{\sup_{n=0,\dots,N} X_n \leq 0} = - \log F(0) = \lambda_\beta.
\]
\end{proof}
\begin{remark}
Proposition~\ref{prop:int_eq} shows that $\exp(-\lambda_\beta)$ is an eigenvalue corresponding to a positive eigenfunction $\varphi$ of the positive bounded linear operator 
\[
   T \colon L^1((-\infty,0]) \to L^1((-\infty,0]), \quad (Tf)(z) := \int_{-\infty}^0 p(y,z) f(y) \, dy, \quad z \leq 0.
\]
One might suspect that $\exp(-\lambda_\beta)$ is the largest spectral value of $T$, i.e.\ $\exp(-\lambda_\beta) = r(T)$ where $r(T)$ denotes the spectral radius of $T$. For instance, such a result holds for positive matrices (by Perron-Frobenius type results, see e.g.\ Corollary I.2.3 in \cite{schaefer:1974}). However, in our case, it can be shown that $r(T) = 1 > \exp(-\lambda_\beta)$. Also one can verify that $r(T)$ is not an eigenvalue of $T$. If $T$ were compact this could not occur, see e.g.\ Theorem V.6.6 in \cite{schaefer:1974}. It remains unclear if $\exp(-\lambda_\beta) \geq \abs{\mu}$ for every other eigenvalue $\mu$ of $T$. Results of this type are known (see e.g.\ Theorem~11.4 in \cite{k-l-s:1989}), but not applicable in our case.
\end{remark}
\section{Universality results}\label{sec:weighted_rw}
\subsection{Polynomial weight functions}
Let $X_1,X_2,\dots$ be a sequence of i.i.d.\ random variables such that $\E{}{X_1} = 0$ and $\E{}{X_1^2}=1$ and $\sigma \colon [0,\infty) \to (0,\infty)$ some measurable function. Let $Z$ denote the corresponding weighted random walk defined in \eqref{def:weighted_RW}. For a sequence $(X_n)_{n \geq 1}$ of standard normal random variables, the problem has already been solved for $\sigma(n) = n^p$. Indeed, the survival exponent is equal to $p+1/2$ in view of \eqref{eq:eq_distr_gaussian_wrw} and Theorem~\ref{thm:polynomial_case_BM} applied to the function $\kappa(\cdot)$ defined by $\kappa(n) = \sigma(1)^2 + \dots + \sigma(n)^2$ and
\[
   \kappa(N) \asymp N^{2p+1}, \qquad \kappa(N+1) - \kappa(N) = \sigma(N+1)^2 \asymp N^{2p}, \quad N \to \infty.
\]
It is a natural question to ask whether the same results holds for any sequence of random variables that obey a suitable moment condition. This is the subject of Theorem~\ref{thm:weighted_rw_lower_bound} and Theorem~\ref{thm:weighted_rw_upper_bound}.
\begin{remark}\label{rem:val_barrier_WRW}
Theorem~\ref{thm:weighted_rw_lower_bound} and Theorem~\ref{thm:weighted_rw_upper_bound} also hold if the barrier $0$ is replaced by any $c \in \mathbb{R}$. The proof of Theorem~\ref{thm:weighted_rw_upper_bound} can be easily modified to cover this case. We briefly indicate below how to adapt the proof of the lower bound. The proofs will be then carried out again for the barrier $1$ instead of $0$. \\
Let $c \in \mathbb{R}$. Take any $x > 0$ such that $\pr{}{X_1 \leq -x} > 0$. Choose $N_0$ such that $-x(\sigma(1) + \dots + \sigma(N_0)) \leq c-1$. On $A_0 := \set{X_1 \leq -x, \dots, X_{N_0} \leq -x}$, it holds that $Z_{N_0} \leq c - 1$ by construction. Then, for $N > N_0$,
\begin{align*}
 \pr{}{\sup_{n=1,\dots,N} Z_n \leq c} &\geq \pr{}{ A_0, \sup_{n=N_0+1,\dots,N} Z_n - Z_{N_0} \leq 1} \\
 & = \pr{}{A_0} \pr{}{\sup_{n=1,\dots,N-N_0} \sum_{k=1}^n \sigma(k+N_0) X_k \leq 1}.
\end{align*}
Hence, it suffices to prove a lower bound for the survival probability of the weighted random walk $\tilde{Z}$ with $\tilde{\sigma}(k) := \sigma(k+N_0)$ ($k \geq 1$) and the barrier $1$ since $\tilde{\sigma}(N) \asymp \sigma(N) \asymp N^p$.
\end{remark}
\subsubsection{Lower bound via Skorokhod embedding}
Here we prove the lower bound of Theorem~\ref{thm:overview_2} under slightly weaker assumptions.
\begin{theorem}\label{thm:weighted_rw_lower_bound}
	Let $(X_n)_{n \geq 1}$ be a sequence of i.i.d.\ centered random variables such that $\E{}{X_1^2} = 1$. Denote by $Z = (Z_n)_{n \geq 1}$ the corresponding weighted random walk defined in \eqref{def:weighted_RW}. Let $\sigma(N) \asymp N^p$ for some $p > 0$. Assume that $\E{}{\abs{X_1}^{\alpha}} < \infty$ for some $\alpha > 4p+2$.  Then 
   \[
      \pr{}{\sup_{n=1,\dots,N} Z_n \leq 0 } \succsim  N^{-(p+1/2)}, \quad N \to \infty.
   \]
\end{theorem}
\begin{proof}
\textbf{Step 1:} Since the $X_i$ are independent centered random variables, $Z$ is a martingale, and one can use a Skorokhod embedding: there exists a Brownian motion $B$ and an increasing sequence of stopping times $(\tau(n))_{n \in \mathbb{N}}$ such that $(Z_n)_{n \in \mathbb{N}}$ and $(B_{\tau(n)})_{n \in \mathbb{N}}$ have the same finite dimensional distributions. Moreover, 
\[
	\E{}{\tau(N)} = \E{}{B_{\tau(N)}^2} = \E{}{Z_N^2} = \sum_{k=1}^N \sigma(k)^2 =: \kappa(N),
\]
see e.g.\ Proposition 11.1 in the survey on the Skorokhod problem of \cite{obloj:2004}. In particular, this implies that $(B_{t \wedge \tau(n)})_{t \geq 0}$ is uniformly integrable.\\
From the contruction of the stopping times described in the cited article (Section 11.1), one deduces that the increments of $(\tau(n))_{n \geq 1}$ are independent since those of $Z$ are.\\
Note that there exist constants $c_1,c_2 > 0$ such that $c_1 N^{2p+1} \leq \kappa(N) \leq c_2 N^{2p+1}$ for all $N$ sufficiently large. W.l.o.g.\ assume that $c_2$ is so large that the upper bound holds for all $N$. Then one has for $\epsilon > 0$ and $N$ large enough
\begin{align}\label{eq:proof_monroe_embedding}\notag
   \pr{}{\sup_{n=1,\dots,N} Z_n \leq 1 } &= \pr{}{\sup_{n=1,\dots,N} B_{\tau(n)} \leq 1 } \\ 
&\geq \pr{}{\sup_{t \in [0,(1+ \epsilon) \kappa(N)]} B_t \leq 1, \, \tau(N) \leq (1 + \epsilon) \kappa(N)} \\ \notag
&\geq \pr{}{\sup_{t \in [0,(1+ \epsilon)c_2 N^{2p+1}]} B_t \leq 1} - \pr{}{\tau(N) - \kappa(N) > \epsilon c_1  N^{2p+1}}.
\end{align}
Clearly,
\begin{equation}\label{eq:proof-monroe-embed1}
   \pr{}{\sup_{t \in [0,(1 + \epsilon)c_2 N^{2p+1}]} B_t \leq 1} \sim \sqrt{\frac{2}{\pi(1+ \epsilon)c_2}} \, N^{-(p+1/2)}, \quad N \to \infty.
\end{equation}
The second term in \eqref{eq:proof_monroe_embedding} may be estimated with Chebychev's inequality if one can control the centered moments of the stopping times $\tau(N)$. Concretely, we claim that for all $N$ and $\gamma \geq 2$ such that $\E{}{\abs{X_1}^{2\gamma}} < \infty$, it holds that
\begin{equation}\label{ineq:moment_monroe_embed}
	\E{}{\abs{\tau(N) - \kappa(N)}^\gamma} = \E{}{\abs{\tau(N) - \E{}{\tau(N)}}^\gamma} \leq C N^{(2p +1/2)\gamma},
\end{equation}
where $C > 0$ is some constant depending only on $\gamma$. If \eqref{ineq:moment_monroe_embed} is true, Chebychev's inequality yields
\begin{align*}
   \pr{}{\tau(N) - \kappa(N) > \epsilon c_1  N^{2p+1}} &\leq \E{}{\abs{\tau(N) - \kappa(N)}^\gamma} \, (\epsilon c_1)^{-\gamma} \,N^{-\gamma (2p + 1)} \\
&\leq  C\,(c_1 \epsilon)^{-\gamma} N^{-\gamma/2}.
\end{align*}
By choosing $\gamma > 2p + 1$, this term is of lower order than $N^{-(p+1/2)}$. The assertion of the proposition follows from \eqref{eq:proof_monroe_embedding}, \eqref{eq:proof-monroe-embed1}, and Remark~\ref{rem:val_barrier_WRW}.\\
\textbf{Step 2:} It remains to verify the validity of \eqref{ineq:moment_monroe_embed}. Choose $\gamma > 2p + 1$ such that $\E{}{\abs{X_1}^{2\gamma}} < \infty$. Since $(B_{\tau(n) \wedge t})_{t \geq 0}$ is uniformly integrable, we deduce from the Burkholder-Davis-Gundy (BDG) inequality (see Proposition 2.1 of \cite{obloj:2004})  that
\[
   \E{}{ \tau(n)^\gamma } \leq C(\gamma) \E{}{ \abs{B_{\tau(n)}}^{2\gamma} } = C(\gamma) \E{}{\abs{Z_n}^{2\gamma}} < \infty.
\]
The finiteness of the last expectation follows from our choice of $\gamma$ and the assumption $\E{}{\abs{X_1}^{2\gamma}} < \infty$. This shows that $\tau(n)^\gamma$ is integrable. \\
Recall that 
\[
	\tilde{B} = (B_{t + \tau(n-1)} - B_{\tau(n-1)})_{t \geq 0}
\]
 is a Brownian motion w.r.t.\ the filtration $\mathbb{G}^{(n)} = (\mathcal{G}^{(n)}_t)_{t \geq 0} := (\mathcal{F}_{t + \tau(n-1)})_{t \geq 0}$ if $B$ is a Brownian motion w.r.t.\ $(\mathcal{F}_t)_{t \geq 0}$. Note that $\tau(n) - \tau(n-1)$ is a $\mathbb{G}^{(n)}$-stopping time for all $n$. Using again the BDG inequality, we get
\begin{align}
   \E{}{(\tau(n) - \tau(n-1))^\gamma} &\leq c_\gamma \, \E{}{\abs{\tilde{B}_{\tau(n) - \tau(n-1)}}^{2\gamma}} \notag \\
&= c_\gamma \, \E{}{\abs{B_{\tau(n)} - B_{\tau(n-1)}}^{2\gamma}} \notag \\
&= c_\gamma \, \E{}{\abs{Z_n - Z_{n-1}}^{2\gamma}} = c_\gamma \, \E{}{\abs{X_1}^{2\gamma}} \sigma(n)^{2\gamma}, \label{ineq:incr_monroe_stop_times}
\end{align}
where $c_\gamma$ is a constant depending on $\gamma$ only and $\E{}{\abs{X_1}^{2\gamma}} < \infty$ by assumption.
For $n = 1,2,\dots$, let 
\[
   Y_n := \tau(n) - \tau(n-1) - \E{}{\tau(n) - \tau(n-1)} = \tau(n) - \tau(n-1) - \sigma(n)^2.
\]
As remarked at the beginning of the proof, the $Y_i$ are independent centered random variables. Using the Marcinkiewicz-Zygmund inequality (or the BDG-inequality), we get
\begin{align*}
   \E{}{\abs{\tau(N) - \kappa(N)}^\gamma} &= \E{}{\abs{\sum_{n=1}^N Y_n}^\gamma} \leq C(\gamma) \E{}{\left(\sum_{n=1}^N Y_n^2 \right)^{\gamma/2}} \\
&= C(\gamma) \norm{\sum_{n=1}^N Y_n^2}_{\gamma/2}^{\gamma/2},
\end{align*}
where $C(\gamma)$ is again some constant that depends only on $\gamma$ and $\norm{\cdot}_p$ denotes the $L^p$-norm (here we need that $\gamma \geq 2$). An application of the triangle inequality yields
\[
   \left( \E{}{\abs{\tau(N) - \kappa(N)}^\gamma} \right)^{2/\gamma} \leq C(\gamma)^{2/\gamma} \sum_{n=1}^N \norm{Y_n^2}_{\gamma/2} = C(\gamma)^{2/\gamma} \sum_{n=1}^N \left( \E{}{ \abs{Y_n}^\gamma} \right)^{2/\gamma}.
\]
Clearly $\abs{Y_n}^\gamma \leq 2^{\gamma}( \abs{\tau(n) - \tau(n-1)}^\gamma + \sigma(n)^{2\gamma})$ implying that
\begin{align*}
   \E{}{\abs{\tau(N) - \kappa(N)}^\gamma}^{2/\gamma} &\leq 4C(\gamma)^{2/\gamma}  \sum_{n=1}^N \left( \E{}{ \abs{\tau(n) - \tau(n-1)}^\gamma}  + \sigma(n)^{2\gamma} \right)^{2/\gamma} \\
&\leq 4C(\gamma)^{2/\gamma}  \sum_{n=1}^N \left( (  c_\gamma \E{}{\abs{X_1}^{2\gamma}} + 1) \sigma(n)^{2\gamma}\right)^{2/\gamma} \\
&\leq 4\left\lbrace C(\gamma) (  c_\gamma \E{}{\abs{X_1}^{2\gamma}} + 1)\right\rbrace^{2/\gamma}  \sum_{n=1}^N \sigma(n)^{4}.
\end{align*}
In the above estimates, the second inequality follows from \eqref{ineq:incr_monroe_stop_times}. We finally arrive at 
\begin{align*}
	\E{}{\abs{\tau(N) - \kappa(N)}^\gamma} &\leq  2^{\gamma} C(\gamma) (  c_\gamma \E{}{\abs{X_1}^{2\gamma}} + 1) \left(\sum_{n=1}^N \sigma(n)^4 \right)^{\gamma/2} \\
&\leq 2^{\gamma} C(\gamma) (  c_\gamma \E{}{\abs{X_1}^{2\gamma}} + 1) \, c_2^{2\gamma} \, N^{(4p+1)\gamma/2},
\end{align*}
proving \eqref{ineq:moment_monroe_embed} with $C = 2^{\gamma} C(\gamma) (  c_\gamma \E{}{\abs{X_1}^{2\gamma}} + 1) c_2^{2\gamma}$.
\end{proof}

\subsubsection{Upper bound via coupling}
The upper bound in Theorem~\ref{thm:overview_2} is a consequence of the following more precise statement.
\begin{theorem}\label{thm:weighted_rw_upper_bound}
   Let $(X_n)_{n \geq 1}$ be a sequence of i.i.d.\ centered random variables such that $\E{}{X_1^2} = 1$. Denote by $Z = (Z_n)_{n \geq 1}$ the corresponding random walk defined in \eqref{def:weighted_RW} and assume that $\E{}{e^{a \abs{X_1} }} < \infty$ for some $a > 0$. Let $\sigma$ be increasing such that $\sigma(N) \asymp N^p$ for some $p > 0$. Then for any $\rho > 4p + 2$
\[
   \pr{}{\sup_{n=1,\dots,N} Z_n \leq 0} \precsim N^{-(p+1/2)} (\log N)^{\rho /2}, \quad N \to \infty.
\]
\end{theorem}
\begin{proof}
   Let $\tilde{Z}_n := \sum_{k=1}^n \sigma(k) \tilde{X_k}$ where the $\tilde{X}_k$ are independent standard normal random variables constructed on the same probability space as the $X_k$. As usual, denote by $S_n = X_1 + \dots + X_n$ the corresponding random walk and define $\tilde{S}$ analogously. Let 
\[
   E_N := \set{ \sup_{n=1,\dots,N} \abs{S_n - \tilde{S}_n} \leq C \log N}
\]
for some constant $C > 0$ to be specified later. We now use a coupling of the random walks $S$ and $\tilde{S}$ that allows us to work with the Gaussian process $\tilde{Z}$ instead of the original process $Z$. Since $\E{}{e^{a \abs{X_1} }} < \infty$ for some $a > 0$, we may assume by Theorem~1 of \cite{k-m-t:1976} that the sequences $(X_n)_{n \geq 1}$ and $(\tilde{X}_n)_{n \geq 1}$ are constructed on a common probability space such that for all $N$ and some $C > 0$ sufficiently large
\begin{equation}\label{ineq:kmt_coupling}
   \pr{}{E_N^c} = \pr{}{\sup_{n=1,\dots,N} \abs{S_n - \tilde{S}_n} > C \log N} \leq K \, N^{-(p+1/2)}
\end{equation}
where $K$ is a constant that depends only on the distribution of $X_1$ and $C$. \\
On $E_N$ one has in view of Abel's inequality (see Lemma 2.1 in \cite{shao:1995} and recall that $\sigma(\cdot)$ is increasing) that for all $n \leq N$ 
\begin{align}
   \sup_{k=1,\dots,n} \abs{Z_k - \tilde{Z}_k} &= \sup_{k=1,\dots,n} \abs{\sum_{j=1}^n \sigma(j) (X_j - \tilde{X}_j)} \\ \notag
&\leq 2 \, \sigma(n) \sup_{k=1,\dots,n} \abs{S_k - \tilde{S}_k} 
\leq 2\, C\, \sigma(n) \log N .
\end{align}
Therefore, on $E_N \cap \set{\sup_{n=1,\dots,N} Z_n \leq 1}$, one has
\[
   \tilde{Z}_n = \tilde{Z_n} - Z_n + Z_n \leq 2 C \sigma(n) \log N + 1, \quad n \leq N.
\]
We may now estimate
\begin{align*}
   &\pr{}{\sup_{n=1,\dots,N} Z_n \leq 1} \leq \pr{}{\sup_{n=1,\dots,N} Z_n \leq 1, E} + \pr{}{E_N^c} \\
&\quad \leq \pr{}{\sup_{n=1,\dots,N} \tilde{Z}_n -  2 C \sigma(n) \log N \leq 1} + \pr{}{E_N^c}.
\end{align*}
In view of \eqref{ineq:kmt_coupling}, the term $\pr{}{E_N^c}$ is at most of order $N^{-(p+1/2)}$. It remains to show that the order of the first term is $N^{-(p+1/2)} (\log N)^{\rho /2}$ for $\rho > 4p + 2$. Let $\kappa(n) := \sigma(1)^2 + \dots + \sigma(n)^2$. If $B$ is a Brownian motion, one has in view of \eqref{eq:eq_distr_gaussian_wrw} that
\[
   \pr{}{\sup_{n=1,\dots,N} \tilde{Z}_n -  2 C \sigma(n) \log N \leq 1 } = \pr{}{\sup_{n=1,\dots,N} B_{\kappa(n)} - 2 C \sigma(n) \log N \leq 1}.
\]
One can now proceed similarly to the proof of Theorem~\ref{thm:polynomial_case_BM}. Note that
\begin{align*}
  &\bigcap_{n=1}^N \set{B_{\kappa(n)} - 2 C \sigma(n) \log N \leq 1} \subseteq \bigcap_{n=1}^{N-1} \set{\sup_{t \in [\kappa(n),\kappa(n+1)]} B_t - 3 C \sigma(n) \log N \leq 1}   \\
&\quad \cup \bigcup_{n=1}^{N-1}\set{\sup_{t \in [\kappa(n),\kappa(n+1)]} B_t - B_{\kappa(n)} > C \sigma(n) \log N } =: G_N \cup H_N.
\end{align*}
Clearly,
\[
   \pr{}{H_N} \leq \sum_{n=1}^{N-1} \pr{}{\sup_{t \in [0,1]} B_t > \frac{C \sigma(n) \log N}{\sqrt{\kappa(n+1) - \kappa(n)}}} \leq N \, \pr{}{\sup_{t \in [0,1]} B_t > \tilde{C} \log N}
\]
where $\tilde{C} = C \inf \set{ \sigma(n) / \sigma(n+1) : n \geq 1} \in (0,C)$ since $\sigma(\cdot)$ is increasing and $\sigma(n) \asymp n^p$. It is easy to show that the last term of the preceding inequality is $o(N^{-\alpha})$ for any $\alpha > 0$, see the proof of Theorem~\ref{thm:polynomial_case_BM}.\\
It remains to estimate $\pr{}{G_N}$. Set $c_1 = \inf \set{ \kappa(n) / n^{2p+1} : n \geq 1} \in (0, \infty)$ since $\sigma(n) \asymp n^p$ and $\sigma(n) > 0$ for all $n \geq 1$ by monotonicity. Hence, $\kappa(n) \geq c_1 n^{2p+1}$ and $t \geq \kappa(n)$ implies that $(t/c_1)^{1/(2p+1)} \geq n$ and therefore,
\begin{align*}
   \pr{}{G_N} &\leq \pr{}{ \bigcap_{n=1}^{N-1} \set{\sup_{t \in [\kappa(n),\kappa(n+1)]} B_t - 3 C \sigma\left( (t/c_1)^{1/(2p+1)}\right) \log N \leq 1} } \\
&\leq  \pr{}{  \sup_{t \in [\kappa(1),\kappa(N)]} B_t - c_2  t^{p/(2p+1)} \log N \leq 1 }.
\end{align*}
Choose $\rho > 2(2p + 1)$, i.e.\ $1/\rho + p/(2p + 1) < 1/2$. Then $t^{p/(2p+1)} \log N \leq t^{p/(2p+1) + 1/\rho}$ for $t \geq (\log N)^{\rho}$ and 
\[
   \pr{}{G_N} \leq  \pr{}{ \sup_{t \in [(\log N)^\rho,\kappa(N)]} B_t - c_2  t^{p/(2p+1) + 1/\rho}  \leq 1 }.
\]
By Slepian's inequality, we have
\begin{align*}
   \pr{}{G_N} &\leq  \frac{\pr{}{ \sup_{t \in [0,\kappa(N)]} B_t - c_2  t^{p/(2p+1) + 1/\rho} \leq 1 }}{\pr{}{ \sup_{t \in [0,(\log N)^\rho]} B_t - c_2  t^{p/(2p+1) + 1/\rho} \leq 1} }.
\end{align*}
As already remarked in the proof of Theorem~\ref{thm:polynomial_case_BM}, the results of \cite{uchiyama:1980} imply that adding a drift of order $t^\alpha$ ($\alpha < 1/2$) to a Brownian motion does not change the rate $T^{-1/2}$. Since $p/(2p+1) + 1/\rho < 1/2$ by the choice of $\rho$, this implies that
\begin{align*}
   \pr{}{G_N} &\precsim \frac{\pr{}{  \sup_{t \in [0,\kappa(N)]} B_t \leq 1} }{ \pr{}{ \sup_{t \in [0,(\log N)^\rho]} B_t \leq 1} } \asymp \kappa(N)^{-1/2} \, (\log N)^{\rho/2} \\
&\asymp N^{-(p+1/2)} (\log N)^{\rho/2}.
\end{align*}
\end{proof}
 \begin{remark}
    We applied the Koml\'{o}s-Major-Tusn\'{a}dy coupling to the random walk $S$ whose increments $X_i$ are i.i.d. If the $X_i$ are independent, but not necessarily identically distributed, one could use the coupling for non-i.i.d.\ random variables introduced by \cite{sakhanenko:1984}: 
\begin{theorem}\label{thm:sakhanenko_coupling}(\cite{sakhanenko:1984}) Assume that the $X_n$ are independent centered random variables and that there is $\lambda > 0$ such that for all $n$
\begin{equation}\label{ineq:sakhanenko_coupling}
   \lambda \, \E{}{e^{\lambda X_n} \abs{X_n}^3} \leq \E{}{X_n^2}. 
\end{equation}
Then for some absolute constant $A >0$
\[
      \pr{}{\sup_{n=1,\dots,N} \abs{S_n - \tilde{S}_n} > C \log N} \leq \left(1 + \lambda \sum_{n=1}^N \E{}{X_n^2}\right) \, N^{-\lambda A  \,  C}, \quad N \geq 1.
\]
\end{theorem}
In particular, under the assumptions of Theorem~\ref{thm:sakhanenko_coupling}, we can control the term $\pr{}{E_N^c}$ in the proof above as before. \\
Note that one can find $\lambda >0$ such that \eqref{ineq:sakhanenko_coupling} is satisfied if the $X_n$ are uniformly bounded or i.i.d.\ such that $\E{}{e^{\lambda_0 \abs{X_1}}} < \infty$ for some $\lambda_0 > 0$. Moreover, assume that \eqref{ineq:sakhanenko_coupling} holds for some $\lambda > 0$. Then 
\[
   \lambda \, \left(\E{}{X_n^2}\right)^{3/2} \leq \lambda \E{}{\abs{X_n}^3} \leq \lambda \E{}{e^{\lambda \abs{X_n}} \abs{X_n}^3} \leq \E{}{X_n^2},
\]
i.e.\ $0 < \lambda \leq \left( \E{}{X_n^2} \right)^{-1/2}$ for all $n$ implying that $(\E{}{X_n^2})_{n \geq 1}$ is necessarily bounded.
 \end{remark}
\subsection{Exponential weight functions}
In this section, we briefly comment on the case of an exponential weight function, i.e.\ $\sigma(n) = e^{\beta n}$ for some $\beta > 0$. The situation here is completely different compared to the polynomial case. \\
First of all, the rate of decay for the discretized process and for the continuous time process is not the same in general. This was observed already in the Brownian case where 
\[
   \pr{}{\sup_{0 \leq t \leq N} B(e^{\beta t}) \leq 0} \sim \frac{1}{\pi} \, e^{-\beta N /2}, \quad N \to \infty,
\]
in view of \eqref{eq:OU_sup_distr} and the fact that $(e^{-\beta t /2}B(e^{\beta t}))_{t \geq 0}$ is an Ornstein-Uhlenbeck process. In particular, for $\beta > 2 \log 2$, the decay is faster than $2^{-N}$ which is a universal lower bound in the discrete framework (cf. \eqref{eq:one_half_to_N}).\\
Secondly, the universality of the survival exponent that one observes in the polynomial case no longer persists even under the assumption of exponential moments as the following example shows. 
\begin{example}
   Let $\sigma(n) = \exp(\beta n)$ for some $\beta \geq \log 2$ and assume that $\pr{}{X_n = 1} = \pr{}{X_n = -1} =1/2$ for all $n$. Then for all $N \geq 1$
\begin{equation}\label{eq:exp_bernoulli_example}
   \sup_{n=1,\dots,N} Z_n \leq 0 \quad \Longleftrightarrow \quad X_1 = \dots = X_N = -1.
\end{equation}
The implication ``$\Leftarrow$'' is trivial. On the other hand, if $X_1 = \dots = X_{k-1} = -1$ and $X_k = 1$, for some $k \leq N$, then
\[
   Z_k = -\sum_{j=1}^{k-1} e^{\beta j} + e^{\beta k} = e^{\beta (k-1)} \frac{e^{\beta} - 2 + e^{\beta (2 - k)} }{e^\beta - 1} > 0
\]
since $\beta \geq \log 2$. This proves the implication ``$\Rightarrow$''. \\
Note that \eqref{eq:exp_bernoulli_example} implies that $\pr{}{\sup_{n=1,\dots,N} Z_n \leq 0} = 2^{-N} = \exp(- \log(2) \, N)$. If we consider $(B(e^{\beta n}))_{n \geq 0}$, the corresponding survival probability is strictly greater than $2^{-N}$ by Lemma \ref{lem:lower_bound_exp_BM}. To be very precise, we actually have to consider $(B_{\kappa(n)})_{n \geq 1}$ where 
\[
   \kappa(n) = \sum_{k=1}^n \sigma(n)^2 = e^{2\beta} \,  \frac{e^{2\beta n} - 1}{e^{2\beta} - 1}.
\]
In particular,
\[
   \pr{}{\sup_{n=1,\dots,N} B_{\kappa(n)} \leq 0} = \pr{}{\sup_{n=1,\dots,N} B(e^{2\beta} - 1) \leq 0}
\]
and the same arguments used in Lemma \ref{lem:lower_bound_exp_BM} show that 
\[
   \lim_{N \to \infty} \frac{1}{N} \log \pr{}{\sup_{n=1,\dots,N} B_{\kappa(n)} \leq 0} > - \log 2.
\]
\end{example}
\bibliographystyle{abbrvnat}
\bibliography{biblio}

\begin{thebibliography}{33}
\providecommand{\natexlab}[1]{#1}
\providecommand{\url}[1]{\texttt{#1}}
\expandafter\ifx\csname urlstyle\endcsname\relax
  \providecommand{\doi}[1]{doi: #1}\else
  \providecommand{\doi}{doi: \begingroup \urlstyle{rm}\Url}\fi

\bibitem[Aurzada and Dereich(2011+)]{aurzada-dereich:2009}
F.~Aurzada and S.~Dereich.
\newblock Universality of the asymptotics of the one-sided exit problem for
  integrated processes.
\newblock \emph{to appear in Ann. Henri Poincar\'{e}}, 2011+.

\bibitem[Beekman(1975)]{beekman:1977}
J.~Beekman.
\newblock Asymptotic distributions for the {O}rnstein-{U}hlenbeck process.
\newblock \emph{J. Appl. Probab.}, 12:\penalty0 107--114, 1975.

\bibitem[Bertoin(1996)]{bertoin:1996}
J.~Bertoin.
\newblock \emph{L\'{e}vy processes}.
\newblock Cambridge University Press, Cambridge (UK), 1996.

\bibitem[Bertoin(1998)]{bertoin:1998}
J.~Bertoin.
\newblock The inviscid burgers equation with {B}rownian initial velocity.
\newblock \emph{Comm. Math. Phys.}, 193\penalty0 (2):\penalty0 397--406, 1998.

\bibitem[Bramson(1978)]{bramson:1978}
M.~Bramson.
\newblock Maximal displacement of branching {B}rownian motion.
\newblock \emph{Comm. Pure Appl. Math}, 31:\penalty0 531--581, 1978.

\bibitem[Dembo and Gao(2011)]{dembo-gao:2011}
A.~Dembo and F.~Gao.
\newblock Persistence of iterated partial sums.
\newblock \emph{Preprint}, 2011.

\bibitem[Doney(2007)]{doney}
R.~A. Doney.
\newblock \emph{Fluctuation theory for {L}\'evy processes}, volume 1897 of
  \emph{Lecture Notes in Mathematics}.
\newblock Springer, Berlin, 2007.

\bibitem[Feller(1970)]{feller-vol2-1970}
W.~Feller.
\newblock \emph{An Introduction to Probability Theory and Its Applications,
  Vol. II, Second Edition}.
\newblock Wiley \& Sons, Inc., New York, 1970.

\bibitem[Goldman(1971)]{goldman:1971}
M.~Goldman.
\newblock On the first passage of the integrated {W}iener process.
\newblock \emph{Ann. Math. Statist.}, 42\penalty0 (6):\penalty0 2150--2155,
  1971.

\bibitem[Gradshteyn and Ryzhik(2000)]{gradshteyn-ryzhik:2000}
I.~Gradshteyn and M.~Ryzhik.
\newblock \emph{Table of integrals, series, and products. Sixth Edition.}
\newblock Academic Press, San Diego, 2000.

\bibitem[Grimmett and Stirzaker(2001)]{grimmett-stirzaker:2001}
G.~Grimmett and D.~Stirzaker.
\newblock \emph{One thousand exercises in probability}.
\newblock Oxford University Press, Oxford, 2001.

\bibitem[Isozaki and Watanabe(1994)]{isozaki-watanabe:1994}
Y.~Isozaki and S.~Watanabe.
\newblock An asymptotic formula for the {K}olmogorov diffusion and a refinement
  of {S}inai's estimates for the integral of {B}rownian motion.
\newblock \emph{Proc. Japan Acad. Ser. A Math.}, 70\penalty0 (9):\penalty0
  271--276, 1994.

\bibitem[Koml\'{o}s et~al.(1976)Koml\'{o}s, Major, and Tusn\'{a}dy]{k-m-t:1976}
J.~Koml\'{o}s, P.~Major, and G.~Tusn\'{a}dy.
\newblock An approximation of partial sums of independent {RV}'s, and the
  sample {DF}. {II}.
\newblock \emph{Probab. Theory Related Fields}, 34\penalty0 (1):\penalty0
  33--58, 1976.

\bibitem[Krasnosel'skii et~al.(1989)Krasnosel'skii, Lifshits, and
  Sobolev]{k-l-s:1989}
M.~Krasnosel'skii, J.~Lifshits, and A.~Sobolev.
\newblock \emph{Positive linear systems - the method of positive operators}.
\newblock Heldermann Verlag, Berlin, 1989.

\bibitem[Larralde(2004)]{larralde:2004}
H.~Larralde.
\newblock A first passage time distribution for a discrete version of the
  {O}rnstein-{U}hlenbeck process.
\newblock \emph{J. Phys. A}, 37:\penalty0 3759--3767, 2004.

\bibitem[Li and Shao(2002)]{li-shao:2002}
W.~Li and Q.~Shao.
\newblock A normal comparison inequality and its applications.
\newblock \emph{Probab. Theory Related Fields}, 122:\penalty0 494--508, 2002.

\bibitem[Li and Shao(2004)]{li-shao:2004}
W.~Li and Q.~Shao.
\newblock Lower tail probabilities for {G}aussian processes.
\newblock \emph{Annals of Probability}, 32\penalty0 (1A):\penalty0 216--242,
  2004.

\bibitem[Majumdar(1999)]{majumdar:1999}
S.~Majumdar.
\newblock Persistence in nonequlilibrium systems.
\newblock \emph{Current Science}, 77\penalty0 (3):\penalty0 370--375, 1999.

\bibitem[{McKean, Jr.}(1963)]{mckean:1963}
H.~{McKean, Jr.}
\newblock A winding problem for a resonator driven by a white noise.
\newblock \emph{J. Math. Kyoto Univ.}, 2:\penalty0 227--235, 1963.

\bibitem[Molchan(1999{\natexlab{a}})]{molchan:1999}
G.~Molchan.
\newblock On the maximum of a fractional {B}rownian motion: Probabilities of
  small values.
\newblock \emph{Teor. Veroyatn. Primen.}, 44\penalty0 (1):\penalty0 111--115,
  1999{\natexlab{a}}.

\bibitem[Molchan(1999{\natexlab{b}})]{molchan:1999a}
G.~Molchan.
\newblock Maximum of a fractional {B}rownian motion: Probabilities of small
  values.
\newblock \emph{Comm. Math. Phys.}, 205\penalty0 (1):\penalty0 97--111,
  1999{\natexlab{b}}.

\bibitem[Molchan and Khokhlov(2004)]{molchan-khokhlov:2004}
G.~Molchan and A.~Khokhlov.
\newblock Small values of the maximum for the integral of fractional {B}rownian
  motion.
\newblock \emph{J. Stat. Phys.}, 114\penalty0 (3-4):\penalty0 924--946, 2004.

\bibitem[Novikov and Kordzakhia(2008)]{novikov-kordzakhia:2008}
A.~Novikov and N.~Kordzakhia.
\newblock Martingales and first passage times of {$\rm AR(1)$} sequences.
\newblock \emph{Stochastics}, 80\penalty0 (2-3):\penalty0 197--210, 2008.

\bibitem[Obl\'{o}j(2004)]{obloj:2004}
J.~Obl\'{o}j.
\newblock The {S}korokhod embedding problem and its offspring.
\newblock \emph{Probab. Surv.}, 1:\penalty0 321--392, 2004.

\bibitem[Sakhanenko(1984)]{sakhanenko:1984}
A.~Sakhanenko.
\newblock Rate of convergence in the invariance principle for variables with
  exponential moments that are not identically distributed.
\newblock In \emph{Trudy Inst. Mat. SO AN SSSR 3, Nauka}, pages 4--49,
  Novosibirsk, 1984.

\bibitem[Sato(1977)]{sato:1977}
S.~Sato.
\newblock Evaluation of the first-passage time probability to a square root
  boundary for the {W}iener process.
\newblock \emph{J. Appl. Probab.}, 14:\penalty0 850--856, 1977.

\bibitem[Schaefer(1974)]{schaefer:1974}
H.~Schaefer.
\newblock \emph{Banach lattices and positive operators}, volume 215 of
  \emph{Die Grundlehren der mathematischen Wissenschaften in
  Einzeldarstellungen}.
\newblock Springer Verlag, New York, Heidelberg, Berlin, 1974.

\bibitem[Shao(1995)]{shao:1995}
Q.-M. Shao.
\newblock Strong approximation theorems for independent random variables and
  their applications.
\newblock \emph{J. Multivariate Anal.}, 52:\penalty0 107--130, 1995.

\bibitem[Simon(2007)]{simon:2007}
T.~Simon.
\newblock The lower tail problem for homogeneous functionals of stable
  processes with no negative jumps.
\newblock \emph{ALEA Lat. Am. J. Prob. Math. Stat.}, 3:\penalty0 165--179,
  2007.

\bibitem[Sinai(1992)]{sinai:1992}
Y.~Sinai.
\newblock Distribution of some functionals of the integral of a random walk.
\newblock \emph{Theoret. and Math. Phys.}, 90:\penalty0 219--241, 1992.

\bibitem[Slepian(1962)]{slepian:1962}
D.~Slepian.
\newblock The one-sided barrier problem for {G}aussian noise.
\newblock \emph{Bell System Techn. J.}, 41:\penalty0 463--501, 1962.

\bibitem[Uchiyama(1980)]{uchiyama:1980}
K.~Uchiyama.
\newblock Brownian first exit from and sojourn over one sided moving boundary
  and application.
\newblock \emph{Z. Wahrscheinlichkeitstheorie verw. Gebiete}, 54:\penalty0
  75--116, 1980.

\bibitem[Vysotsky(2010)]{vysotsky:2010}
V.~Vysotsky.
\newblock On the probability that integrated random walks stay positive.
\newblock \emph{Stochastic Proc. Appl.}, 120:\penalty0 1178--1193, 2010.

\end{thebibliography}
\end{document}